\documentclass{amsart}
\usepackage{amssymb, amsmath, latexsym, amscd, graphicx, tabularx, color}
\usepackage{rotating}
\usepackage[all]{xy}
\usepackage[dvipdfm]{hyperref}
\usepackage[all]{hypcap}

\newtheorem{theorem}{Theorem}[section]
\newtheorem{lemma}[theorem]{Lemma}
\newtheorem{cor}[theorem]{Corollary}
\newtheorem{definition}[theorem]{Definition}

\newtheorem{remark}[theorem]{Remark}
\newtheorem{example}[theorem]{Example}

\def\pagenumber{1}
\begin{document}
\setcounter{page}{\pagenumber}
\parindent=0pt
\def\SRA{\hskip 2pt\hbox{$\joinrel\mathrel\circ\joinrel\to$}}
\def\tbox{\hskip 1pt\frame{\vbox{\vbox{\hbox{\boldmath$\scriptstyle\times$}}}}\hskip 2pt}
\def\circvert{\vbox{\hbox to 8.9pt{$\mid$\hskip -3.6pt $\circ$}}}
\def\IM{\hbox{\rm im}\hskip 2pt}
\def\COIM{\hbox{\rm coim}\hskip 2pt}
\def\COKER{\hbox{\rm coker}\hskip 2pt}
\def\TR{\hbox{\rm tr}\hskip 2pt}
\def\GRAD{\hbox{\rm grad}\hskip 2pt}
\def\RANK{\hbox{\rm rank}\hskip 2pt}
\def\MOD{\hbox{\rm mod}\hskip 2pt}
\def\DEN{\hbox{\rm den}\hskip 2pt}
\def\DEG{\hbox{\rm deg}\hskip 2pt}

\title[Exotic $n$-d'Alembert PDE's and Stability]{\mbox{}\\[1cm] EXOTIC $\mathbf{n}$-D'ALEMBERT PDE's AND STABILITY}

\author{Agostino Pr\'astaro}
\thanks{This paper will be published in the book: \textit{Nonlinear Analysis: Stability, Approximation and Inequalities} - Dedicated to Themistocles M. Rassias for his $60^{th}$ birthday. (Eds. G. Georgiev (USA), P. Pardalos (USA) and H. M. Srivastava (Canada)), Springer, New York.\\
Work partially supported by MIUR Italian grants ''PDE's Geometry and Applications''.}
\maketitle
\vspace{-.5cm}
{\footnotesize
\begin{center}
Department of Methods and Mathematical Models for Applied
Sciences, University of Rome ''La Sapienza'', Via A.Scarpa 16,
00161 Rome, Italy. \\
E-mail: {\tt Prastaro@dmmm.uniroma1.it}
\end{center}
}

\begin{abstract} \noindent
In the framework of the PDE's algebraic topology, previously introduced by A. Pr\'astaro, {\em exotic $n$-d'Alembert PDE's} are considered. These are $n$-d'Alembert PDE's, $(d'A)_n$, admitting Cauchy manifolds $N\subset (d'A)_n$ identifiable with exotic spheres, or such that $\partial N$, can be exotic spheres. For such equations local and global existence theorems and stability theorems are obtained.
\end{abstract}

\noindent {\bf AMS Subject Classification:} 55N22, 58J32, 57R20; 58C50; 58J42; 20H15; 32Q55; 32S20.

\vspace{.08in} \noindent \textbf{Keywords}: d'Alembert PDE's; Integral bordisms in PDE's;
Existence of local and global solutions in PDE's; Conservation laws;
Crystallographic groups; Exotic spheres; Singular Cauchy problems; Stability.

\section{INTRODUCTION}

\rightline{\footnotesize{\it ''Do exotic PDE's exist,}}

\rightline{\footnotesize{\it where exotic $7$-spheres of the same $\Theta_7$-class,}}

\rightline{\footnotesize{\it do not bound smooth solutions ?''}}

\vskip 0.5cm
In some previous works we studied $n$-d'Alembert PDE's  by using the PDE's algebraic topology, introduced by A. Pr\'astaro. (See Refs. \cite{PRA1, PRA4, PRA7, PRA10, PRA-RAS-0, PRA-RAS-1}.) In particular, in \cite{PRA10} are characterized also the stability properties of such equations, showing that the $n$-d'Alembert equation is an extended crystal PDE, for any $n\ge 2$, and obtaining
criteria in order to be an extended $0$-crystal PDE and a $0$-crystal PDE. Furthermore, we
 proved that for any $n\ge 2$ one can canonically associate to the $n$-d'Alembert equation another PDE,
 {\em stable extended crystal $n$-d'Alembert PDE}, having the same regular smooth solutions of the
 $n$-d'Alembert equation, but there, in these solutions, do not occur finite times unstabilities.
 This allowed to avoid all the problems present in the applications, related to finite unstability of solutions.
 Furthermore, we formulated a workable criterion to recognize asymptotic stability suitably averaging
perturbations. (See \cite{PRA7, PRA8, PRA9, PRA10, PRA11}.)

 As for higher dimensions, i.e., when $n\ge 7$, existence of exotic spheres are admitted, it becomes interesting to investigate which implications such phenomena have on the characterization of global solutions of $n$-d'Alembert PDE's and their stability. In some previous papers, A. Pr\'astaro has studied in some details such phenomena for the Ricci flow equation, since this is important to prove the Poincar\'e's conjecture on three dimensional Riemannian manifolds, and its generalizations to higher dimensions. (See \cite{PRA4, PRA12, PRA13, PRA14, PRA15, PRA16, PRA17}.) Furthermore, in \cite{PRA18} generalizations of such phenomena are considered for any PDE and characterized in the framework of Pr\'astaro's  PDE's algebraic topology.

In this paper we aim to apply this theory to exotic $n$-d'Alembert PDE's, and to study the interplay between the geometric stability characterization of such equations by using the algebraic topological methods previously introduced in \cite{PRA7, PRA8, PRA9, PRA10, PRA11, PRA14}. (See also \cite{AG-PRA1, AG-PRA2, PRA-RAS-2}.)

After this Introduction, the paper splits into two more sections. The first devoted to the characterization of exotic $n$-d'Alembert PDE's, and the second to the stability properties of such equations.
The main new result is Theorem \ref{stability-in-exotic-8-d-alembert-pdes} characterizing global solutions of exotic $8$-d'Alembert equation. This theorem allows us to answer in the affirmative to the question put in quotation marks, at the beginning of this Introduction. In fact, after Theorem \ref{stability-in-exotic-8-d-alembert-pdes} we can state that two diffeomorphic exotic $7$-sphere, identified with two Cauchy manifolds in $(d'A)_8$ over $\mathbb{R}^8$, bound singular solutions only - they cannot bound smooth solutions. (Compare with the situation in the Ricci flow equation on compact, simply connected $7$-dimensional Rimennian manifolds \cite{PRA17}.)

\section{EXOTIC $n$-D'ALEMBERT PDE's}\label{exotic-n-d-alembert-pde-section}

In this section we resume some our recent results about the
algebraic topology characterization of PDE's, and that will be
useful in the next section.\footnote{For general informations on
bordism groups, and related problems in differential topology and PDE's geometry, see,
e.g., Refs.\cite{BOARD, GOLD1, GOL-GUIL, GRO, HIRSCH2, KRA-LYC-VIN, LYCH-PRAS, PRA00, PRA01, PRA1, PRA2, PRA3, PRA4, PRA5, PRA6, PRA8, SWITZER, TOGNOLI, WALL1, WALL2, WARNER}. For crystallographic groups see
references quoted in \cite{PRA14}. About differential structures and algebraic topology of exotic spheres, see \cite{BRIESKORN1, BRIESKORN2, DUB-FOM-NOV,  MILNOR1, MOISE1, MOISE2, NASH, PER1, PER2, PRA12, PRA16, PRA17, PRA18, SMALE1}.} In particular let us recall the following
theorem that relates integral bordism groups of PDE's to subgroups
of crystallographic groups. For their proofs we address reader
to the original papers.

\begin{remark}
Here and in the following we shall denote the boundary $\partial V$
of a compact $n$-dimensional manifold $V$, split in the form
$\partial V=N_0\bigcup P\bigcup N_1$, where $N_0$ and $N_1$ are two
disjoint $(n-1)$-dimensional submanifolds of $V$, that are not
necessarily closed, and $P$ is another $(n-1)$-dimensional
submanifold of $V$. For example, if $V=S\times I$, where
$I\equiv[0,1]\subset\mathbb{R}$, one has $N_0=S\times\{0\}$,
$N_1=S\times\{1\}$, $P=\partial S\times I$. In the particular case
that $\partial S=\varnothing $, one has also $P=\varnothing $.  Let us
also recall that with the term {\em quantum solutions} we mean
integral bordisms relating Cauchy hypersurfaces of $E_{k+s}$,
contained in $J^{k+s}_n(W)$, but not necessarily contained into
$E_{k+s}$. (For details see \cite{PRA00, PRA01, PRA1, PRA2, PRA3, PRA4, PRA5, PRA6,
PRA7}.)
\end{remark}

\begin{theorem}{\em\cite{PRA14}}
Bordism groups relative to smooth manifolds can be considered as
extensions of subgroups of crystallographic groups.
\end{theorem}

\begin{definition}
We say that a PDE $E_k\subset J^k_n(W)$ is an {\em extended
0-crystal PDE}, if its integral bordism group is zero.\footnote{Here for integral bordism group we refer to weak integral bordism group $\Omega^{E_k}_{n-1,w}$.}
\end{definition}

The following theorem relates the integrability properties of a PDE
to crystallographic groups.

\begin{theorem}{\em(Crystal structure of PDE's).\cite{PRA14}}\label{crystal structure of PDEs}
Let $E_k\subset J^k_n(W)$ be a formally integrable and completely
integrable PDE, such that $\dim E_k\ge 2n+1$. Then its integral
bordism group $\Omega_{n-1}^{E_k}$ is an extension of a subgroup of
some crystallographic group. In this case, we say that $E_k$ is an {\em extended crystal
PDE} and we define {\em crystal group of $E_k$} the littlest of such crystal groups. The corresponding
dimension will be called {\em crystal dimension of $E_k$}.

Furthermore if $W$ is contractible,
then $E_k$ is an extended $0$-crystal PDE, when $\Omega_{n-1}=0$.
\end{theorem}

In the following we relate crystal structure of PDE's to the
existence of global smooth solutions, identifying an
algebraic-topological obstruction.

\begin{theorem}{\em\cite{PRA14, PRA01, PRA1, PRA2}.}
Let $E_k\subset J^k_n(W)$ be a formally integrable and completely
integrable PDE. Then, in the algebra $\mathbf{H}_{n-1}(E_k)\equiv
Map(\Omega_{n-1}^{E_k};\mathbb{R})$, {\em(Hopf algebra of $E_k$)},
there is a subalgebra, {\em(crystal Hopf algebra)} of $E_k$. On such
an algebra we can represent the algebra $\mathbb{R}^{G(d)}$
associated to the crystal group $G(d)$ of $E_k$. (This justifies the
name.) We call {\em crystal conservation laws} of $E_k$ the elements
of its crystal Hopf algebra.\footnote{Recall that $A\equiv
Map(\Omega,\mathbb{R})$, $\Omega$ a group, has a natural structure
of Hopf algebra if $\Omega$ is a finite group. If $\Omega$ is not
finite group, then $A$ has a structure of Hopf algebra in extended
sense. (See ref.\cite{PRA2}.)}
\end{theorem}

\begin{theorem}{\em\cite{PRA14, PRA8, PRA9, PRA10, PRA11}.}
Let $E_k\subset J^k_n(W)$ be a formally integrable and completely
integrable PDE. Then, the obstruction to find global smooth
solutions of $E_k$ can be identified with the quotient
$\mathbf{H}_{n-1}(E_\infty)/\mathbb{R}^{\Omega_{n-1}}$.

We define {\em crystal obstruction} of $E_k$ the above quotient of
algebras, and put: $ cry(E_k)\equiv
\mathbf{H}_{n-1}(E_\infty)/\mathbb{R}^{\Omega_{n-1}}$. We call
{\em$0$-crystal PDE} one $E_k\subset J^k_n(W)$ such that
$cry(E_k)=0$.\footnote{An extended $0$-crystal PDE $E_k\subset
J^k_n(W)$ does not necessitate to be a $0$-crystal PDE. In fact, in order that
$E_k$ is an extended $0$-crsytal PDE it is enough $\Omega_{n-1,w}^{E_k}=0$.
This does not necessarily imply that $\Omega_{n-1}^{E_k}=0$.}
\end{theorem}

\begin{cor}
Let $E_k\subset J^k_n(W)$ be a $0$-crystal PDE. Let $N_0, N_1\subset
E_k$ be two initial and final Cauchy data of $E_k$ such that
$X\equiv N_0\sqcup  N_1\in[0]\in\Omega_{n-1}$. Then there exists a
smooth solution $V\subset E_k$ such that $\partial V=X$.
\end{cor}

\begin{definition}[Exotic PDE's]\label{exotic-pdes}
Let $E_k\subset J^k_n(W)$ be a $k$-order PDE on the fiber bundle $\pi:W\to M$, $\dim W=m+n$, $\dim M=n$. We say that $E_k$ is an {\em exotic PDE} if it admits Cauchy integral manifolds $N\subset E_k$, $\dim N=n-1$, such that one of the following two conditions is verified.\footnote{In this paper we will use the same notation adopted in \cite{PRA17}: $\thickapprox$ homeomorphism; $\cong$ diffeomorphism; $\approxeq$ homotopy equivalence; $\simeq$ homotopy.}

{\em(i)} $\Sigma^{n-2}\equiv\partial N$ is an exotic sphere of dimension $(n-2)$, i.e. $\Sigma^{n-2}$ is homeomorphic to $S^{n-2}$, ($\Sigma^{n-2}\thickapprox S^{n-2}$) but not diffeomorphic to $S^{n-2}$, ($\Sigma^{n-2}\not\cong S^{n-2}$).

{\em(ii)} $\varnothing=\partial N$ and $N\thickapprox S^{n-1}$, but $N\not\cong S^{n-1}$.
\end{definition}

\begin{example}
The Ricci flow equation can be an exotic PDE for $n$-dimensional Riemannian manifolds of dimension $n\ge 7$. (See \cite{PRA17}.)
\end{example}

\begin{example}
The Navier-Stokes equation can be encoded on the affine fiber bundle $\pi:W\equiv M\times\mathbf{I}\times\mathbb{R}^2\to M$, $(x^\alpha,\dot x^i,p,\theta)_{0\le\alpha\le 3,1\le i\le 3}\mapsto(x^\alpha)$. (See \cite{PRA1}.) Therefore, Cauchy manifolds are $3$-dimensional manifolds. For such dimension do not exist exotic spheres. Therefore, the Navier-Stokes equation cannot be an exotic PDE. Similar considerations hold for PDE's of the classical continuum mechanics.
\end{example}

\begin{example}
Let $M$ be a $n$-dimensional manifold, $n\ge 2$, and let
$\pi:E\equiv M\times{\mathbb R}\to M$ be the trivial vector fiber
bundle on $M$. The {\em $n$-d'Alembert equation}:\footnote{If $n=2$
we say simply d'Alembert equation and we will put $
(d'A)\equiv(d'A)_2$.} $ {{\partial^n logf}\over{\partial
x_1\cdots\partial x_n}}=0$, is a $n$-order closed partial
differential relation (in the sense of Gromov \cite{GRO}) on the
fiber bundle $\pi:E\equiv M\times\mathbb{R}\to M$, i.e., it defines a subset $Z_n\subset
J{\it D}^n(E)$, without boundary, $\partial Z_n=\varnothing $.
Let
$\big\{x^\alpha,u,u_\alpha,u_{\alpha\beta},\dots,u_{{\alpha}_1\cdots\alpha_n}\big\}
$ be a coordinate system on $J{\it D}^n(E)$ adapted to
the fiber structures: $\pi_n:J{\it D}^n(E)\to M,\hskip
3pt \overline{\pi}_{n,0}:J{\it D}^n(E)\to \mathbb{R}$.
Then, $Z_n=F^{-1}(0), F:J{\it D}^n(E)\to \mathbb{R}$,
where $F$ is a sum of terms of the type: $$
F[s;r\vert\alpha,{\beta}_1\beta_2,\cdots,\gamma_1\dots\gamma_q]\equiv
su^ru_\alpha u_{{\beta}_1\beta_2}\cdots u_{\gamma_1\dots\gamma_q},$$
with
$\alpha\not={\beta}_1\not=\beta_2\not=\cdots\not=\gamma_1\not=\dots\not=\gamma_q
\le n, s\in\mathbb{Z}, r\in{\mathbb N}\cup\{0\}$. Furthermore, the
term in $F$ containing $u_{1\cdots n}$ is just $u_{1\cdots
n}u^{n-1}$.\footnote{For example, for $n=2$ one has:
$F=u_{xy}u-u_xu_y$, and for $n=3$ one has:
$F=u_{xyz}u^2-u_{xy}u_zu-u_{xz}u_yu+u_xu_yu_z$.} Note that $F$ has
not locally constant rank on all $Z_n,$ so $Z_n$ is not a
submanifold of $J{\it D}^n(E)$. Furthermore, on the open
subset $C_n\equiv u^{-1}(\mathbb{R}\setminus 0)\subset J{\it
D}^n(E),$ one recognizes that $F$ has locally constant
rank $1$. Hence $Z_n\cap C_n$ is a subbundle of $J{\it
D}^n(E)\to M$, of dimension
$n+{{(2n)!}\over{(n!)^2}}-1$. In the following, for abuse of
notation, we shall denote by $(d'A)_n$ whether $Z_n$ or $Z_n\cap
C_n$.

The $n$-d'Alembert equation over $M=\mathbb{R}^n$ can be an exotic PDE for $n$-dimensional manifolds of dimension $n\ge 8$, but one must carefully consider the meaning of smooth Cauchy $(n-1)$-dimensional manifolds there. In fact, it is not possible for any $n$ to embed in the fiber bundle $E=\mathbb{R}^{n+1}$ exotic $(n-1)$-spheres. To be more precise, let us consider the case $n=8$. Then since $S^7\subset\mathbb{R}^8\subset E$, we can embed in $E$ the standard $7$-dimensional sphere. On the other hand it is well known, after some results by E. V. Brieskorn \cite{BRIESKORN1, BRIESKORN2}, that homotopy $7$-spheres $\Sigma^7$ can be identified with the intersections of complex hypersurfaces $Y_\kappa$, $1\le \kappa\le 28$, in $\mathbb{C}^{5}$, with a $9$-dimensional small sphere $X$, around the singular point at the origin in $Y_\kappa$: $\Sigma^7=Y_\kappa\bigcap X$. See equations in {\em(\ref{seven-exotic-sphere-equations})}.
\begin{equation}\label{seven-exotic-sphere-equations}
\scalebox{0.8}{$\Sigma^7=    \left\{(z_1,\cdots,z_5)\in\mathbb{C}^5\: \left| \begin{array}{l}
                                                 (Y_\kappa):\hskip 3pt    z_1^2+z_2^2+z_3^2+z_4^3+z_5^{6\kappa-1}=0,\hskip 3pt 1\le\kappa\le 28\\
                                                   (X):\hskip 3pt   \sum_j(z_j\bar z_j)= 1\\
                                                   \end{array}\right.\right\}\subset\mathbb{C}^5\cong\mathbb{R}^{10}.$}
\end{equation}
$X\bigcap Y_k$,  $1\le \kappa\le 28$, have the differential structures identified by $\Theta_7\cong\mathbb{Z}_{28}$.\footnote{$\Theta_n$ denotes the additive group of diffeomorphism classes of oriented smooth homotopy spheres of dimension $n$.} In other words exotic $7$-spheres are framed manifolds $\Sigma^7\subset\mathbb{R}^{7+s}$, with $s\ge 3$. Therefore, we cannot embed in the total space $E\equiv\mathbb{R}^9$, of the fiber bundle $\pi:E\to \mathbb{R}^8$, any homotopy $7$-sphere. However, this does not exclude that some smooth Cauchy $7$-dimensional manifolds in $(d'A)_8$ can be identified with exotic $7$-spheres. In fact, since $\dim(d'A)_8=12877$, it is satisfied the (Whitney) condition $\dim(d'A)_8\ge 2\times 7+1=15$ to embed $\Sigma^7$ in $(d'A)_8$. If $N\subset (d'A)_8$ is the image of such an embedding, $N$ cannot in general be diffeomorphic to its image $Y\subset E$, via the canonical projection $\pi_{8,0}:J{\it D}^8(E)\to E$. So, in this case we shall talk about singular Cauchy $7$-manifolds of $(d'A)_8$. Furthermore, let us emphasize that since the equation defining the open PDE $C_8\bigcap(d'A)_8$, can be solved with respect to the coordinate $u_{x^1\cdots x^n}$, we can embed homotopy $7$-spheres $\Sigma^7$ as smooth integral submanifolds $N\subset (d'A)_8\subset J^8_8(E)$, such that their Thom-Boardman singular points should be not frozen singularities in the sense introduced in \cite{LYCH-PRAS}. Therefore, we can state that such $7$-dimensional integral manifolds are contained in $8$-dimensional integral manifolds $V\subset(d'A)_8$, (singular) solutions of $(d'A)$. Such  $7$-dimensional integral manifolds are called {\em admissible Cauchy manifolds} of $(d'A)_8$.
\end{example}

\section{STABILITY IN EXOTIC $n$-D'ALEMBERT PDE's}\label{pdes-spectra-section}

Let us consider, now, the stability of PDE's in the framework of the
geometric theory of PDE's. We shall follow the line just drawn in
some our previous papers on this subject, where we have unified the
integral bordism for PDE's and stability and related the quantum
bordism of PDE's to Ulam stability \cite{ULA}.

\begin{definition}{\em(Singular solutions of PDE's).}
Let $\pi:W\to M$ be a fiber bundle with $\dim M=n$ and $\dim W=m+n$.
Let $E_k\subset J{\it D}^k(W)$ be a PDE and $V\subset E_k$ be a
solution of $E_k$. We say that $p\in V$ is a {\em singular point} of
$V$ of type $\Sigma_i$, $i=0,1,\cdots,n$, if the canonical map
$\pi_k|_V:V\to M$ has a Thom-Boardman singularity of type $S_i$ \cite{BOARD, PRA8}. Let
 $\Sigma(V)\subset V$ be the set of singular points of $V$. Then
$V\setminus\Sigma(V)=\bigcup_rV_r$ is the disjoint union of
connected components $V_r$. For every of such components
$\pi_k:V_r\to M$ is an immersion and can be represented by means of
image of $k$ derivative of some section $s$ of $\pi$, i.e.,
$V_r=D^ks(U_r)$, where $U_r\subset M$ is an open subset of $M$. We
call also solutions of $E_k$ any submanifold $V\subset E_k$ that are
obtained by projection of $\pi_{k+h,k}$ of some solution $V'\subset
(E_k)_{+h}\subset J{\it D}^{k+h}(W)$, represented by a smooth
integral submanifold of $(E_k)_{+h}$, i.e., $V=\pi_{k+h,k}(V')$. In
general a such a solution $V$ is not more represented by a smooth
submanifold of $E_k$. We say also that instead the smooth manifold
$V'\subset (E_k)_{+h}$ {\em solves singularities} of $V$, (or {\em
smooths $V$}). More general solutions are considered taking into account the canonical embedding $J{\it D}^{k}(W)\hookrightarrow J^k_n(W)$, where $J^k_n(W)$ is the $k$-jet space for $n$-dimensional submanifolds of $W$. (For details see \cite{PRA00, PRA01, PRA1, PRA2, PRA4, PRA5}.)

We define {\em weak solutions}, solutions $V\subset E_k$, such that
the set $\Sigma(V)$ of singular points of $V$, contains also
discontinuity points, $q,q'\in V$, with
$\pi_{k,0}(q)=\pi_{k,0}(q')=a\in W$, or $\pi_{k}(q)=\pi_{k}(q')=p\in
M$. We denote such a set by $\Sigma(V)_S\subset\Sigma(V)$, and, in
such cases we shall talk more precisely of {\em singular boundary}
of $V$, like $(\partial V)_S=\partial V\setminus\Sigma(V)_S$.
However for abuse of notation we shall denote $(\partial V)_S$,
(resp. $\Sigma(V)_S$), simply by $(\partial V)$, (resp.
$\Sigma(V)$), also if no confusion can arise.
\end{definition}

\begin{definition}{\em(Stable solutions of PDE's).}
Under the same hypotheses of above definitions, let $X\to E_k$ be a
regular solution, where $X\subset M$ is a smooth $n$-dimensional
compact manifold with boundary $\partial X$. Then $f$ is {\em
stable} if there is a neighborhood $W_f$ of $f$ in
$\underline{Sol}(E_k)$, the manifold of regular solutions of $E_k$,
such that each $f'\in W_f$ is equivalent to $f$, i.e., $f$ is
transformed in $f'$ by vertical symmetries of $E_k$.
\end{definition}

\begin{theorem}\label{deformation}{\em\cite{PRA8}}
Let $E_k\subset J{\it D}^k(W)$ be a $k$-order PDE on the fiber
bundle $\pi:W\to M$ in the category of smooth manifolds, $dim
W=m+n$, $\dim M=n$, $m>1$. Let $s:M\to W$ be a section, solution of
$E_k$, and let $\nu:M\to s^*vTW\equiv E[s]$ be a solution of the
linearized equation $E_k[s]\subset  J{\it D}^k(E[s])$. Then to $\nu$
is associated a flow $\{\phi_\lambda\}_{\lambda\in J}$, where
$J\subset \mathbb{R}$ is a neighborhood of $0\in\mathbb{R}$, that
transforms $V$ into a new solution $\widetilde{V}\subset E_k$.
\end{theorem}

\begin{definition}\label{fun-stable-PDE}
Let $E_k\subset J^k_n(W)$, where $\pi:W\to M$ is a fiber bundle, in
the category of smooth manifolds. We say that $E_k$ is {\em
functionally stable} if for any compact regular solution $V\subset
E_k$, such that $\partial V=N_0\bigcup P \bigcup N_1$ one has
quantum solutions $\widetilde{V}\subset J^{k+s}_n(W)$, $s\ge 0$,
such that $\pi_{k+s,0}(\widetilde{N}_0\sqcup
\widetilde{N}_1)=\pi_{k,0}(N_0\sqcup  N_1)\equiv X\subset W$, where
$\partial \widetilde{V}=\widetilde{N}_0\bigcup
\widetilde{P}\bigcup\widetilde{N}_1$.

We call the set $\Omega[V]$ of such solutions $\widetilde{V}$ the
{\em full quantum situs} of $V$. We call also each element
$\widetilde{V}\in \Omega[V]$ a {\em quantum fluctuation} of
$V$.\footnote{Let us emphasize that to $\Omega[V]$ belong also (non
necessarily regular) solutions $V'\subset E_k$ such that $N_0'\sqcup
N_1'=N_0\sqcup  N_1$, where $\partial V'=N_0'\bigcup P'\bigcup N_1'$.}

We call {\em infinitesimal bordism} of a regular solution $V\subset
E_k\subset J{\it D}^k(W)$ an element $\widetilde{V}\in\Omega[V]$,
defined in the proof of Theorem \ref{deformation}. (See \cite{PRA8}.) We denote by
$\Omega_0[V]\subset \Omega[V]$ the set of infinitesimal bordisms of
$V$. We call $\Omega_0[V]$ the {\em infinitesimal situs} of $V$.
\end{definition}

\begin{definition}\label{fun-stable}
Let $E_k\subset J^k_n(W)$, where $\pi:W\to M$ is a fiber bundle, in
the category of smooth manifolds. We say that a regular solution
$V\subset E_k$, $\partial V=N_0\bigcup P \bigcup N_1$, is {\em
functionally stable} if the infinitesimal situs $\Omega_0[
V]\subset\Omega[V]$ of $V$ does not contain singular infinitesimal
bordisms.
\end{definition}

\begin{theorem}\label{main}{\em\cite{PRA7, PRA8}}
Let $E_k\subset J^k_n(W)$, where $\pi:W\to M$ is a fiber bundle, in
the category of smooth manifolds. If $E_k$ is formally integrable
and completely integrable, then it is functionally stable as well as
Ulam-extended superstable.

A regular solution $V\subset E_k$ is stable iff it is functionally
stable.
\end{theorem}

\begin{remark}
Let us emphasize that  the definition of functionally stable PDE
interprets in pure geometric way the definition of Ulam superstable
functional equation just adapted to PDE's.

\end{remark}

\begin{definition}
We say that $E_k\subset J{\it D}^k(W)$ is a {\em stable extended crystal PDE} if it is an extended crystal PDE that is functionally stable and all its regular smooth solutions are (functionally) stable.

We say that $E_k\subset J{\it D}^k(W)$ is a {\em stabilizable extended crystal PDE} if it is an extended crystal PDE and to $E_k$ can be canonically associated a stable extended crystal PDE ${}^{(S)}E_k\subset J{\it D}^{k+s}(W)$. We call ${}^{(S)}E_k$ just the {\em stable extended crystal PDE of} $E_k$.
\end{definition}

We have the following criteria for functional stability of solutions
of PDE's and to identify stable extended crystal PDE's.
\begin{theorem}{\em(Functional stability criteria).\cite{PRA8}}\label{criteria-fun-stab}
Let $E_k\subset J{\it D}^k(W)$ be a $k$-order formally integrable
and completely integrable PDE on the fiber bundle $\pi:W\to M$,
$\dim W=m+n$, $\dim M=n$.

{\em 1)} If the symbol $g_k=0$, then all the smooth regular
solutions $V\subset E_k\subset J{\it D}^k(W)$ are functionally
stable, with respect to any non-weak perturbation. So $E_k$ is a stable extended crystal.

{\em 2)} If $E_k$ is of finite type, i.e., $g_{k+r}=0$, for $r>0$,
then all the smooth regular solutions $V\subset E_{k+r}\subset J{\it
D}^{k+r}(W)$ are functionally stable, with respect to any non-weak
perturbation. So $E_k$ is a stabilizable extended crystal with stable extended crystal ${}^{(S)}E_k=E_{k+r}$.

{\em 3)} If $V\subset(E_k)_{+\infty}\subset J{\it D}^\infty(W)$ is a
smooth regular solution, then $V$ is functionally stable, with
respect to any non-weak perturbation. So any formally integrable and completely integrable PDE $E_k\subset J{\it D}^k(W)$, is a stabilizable extended crystal, with stable extended crystal ${}^{(S)}E_k=(E_k)_{+\infty}$.

\end{theorem}

\begin{remark}
Let us also remark that in evolutionary PDE's, i.e., PDE's built on
a fiber bundle $\pi:W\to M$, over a ''space-time'' $M$, $\{x^\alpha,
y^j\}_{0\le\alpha\le n, 1\le j\le
m}\mapsto\{x^\alpha\}_{0\le\alpha\le n}$, where $x^0=t$ represents
the time coordinate, one can consider ''asymptotic stability'',
i.e., the behaviour of perturbations of global solutions for
$t\to\infty$. In such cases we can recast our formulation on the
corresponding compactified space-times. (For details see
\cite{PRA8, PRA9}.)

From above results one can see that, in general, the functional
stability of smooth regular solutions is a very strong requirement.
However, above theorems, give us workable criteria to obtain
subequations of $E_k$ whose smooth regular solutions have assured
functional stability.

It is also useful an weaker requirement than functional stability.
This is related to a concept of ''averaged stability''.
\end{remark}

In fact we have the following definition.

\begin{definition}
Let $E_k\subset J{\it D}^k(W)$ be a formally integrable and
completely integrable PDE the fiber bundle $\pi:W\to M$, and let
$V=D^ks(M)\subset E_k$ be a regular smooth solution of $E_k$. Let
$\xi:M\to E_k[s]$ be the general solution of $E_k[s]$. Let us assume
that there is an Euclidean structure on the fiber of $E[s]\to M$.
Then, we say that $V$ is {\em average asymptotic stable} if the
function of time $\mathfrak{p}(t)$ defined by the formula:
\begin{equation}\label{average-square-perturbation}
    \mathfrak{p}(t)=\frac{1}{2 vol(B_t)}\int_{B_t}\xi^2\hskip 3pt \eta
\end{equation}
has the following behaviour:
$\mathfrak{p}(t)=\mathfrak{p}(0)e^{-ct}$ for some real number $c>0$.
We call $\tau_0=1/c_0$ the {\em characteristic stability time} of
the solution $V$. If $\tau_0=\infty$ it means that $V$ is average
instable.\footnote{In the following, if there are not reasons of
confusion, we shall call also stable solution a smooth regular
solution of a PDE $E_k\subset J{\it D}^k(W)$ that is average
asymptotic stable.}
\end{definition}

We have the following criterion of average asymptotic stability.
\begin{theorem}{\em(Criterion of average asymptotic stability).\cite{PRA8}}\label{criterion-average-asymptotic-stability}
A regular global smooth solution $s$ of $E_k$ is average stable if
the following conditions are satisfied:
\begin{equation}\label{stability-inequality}
   \mathop{\mathfrak{p}}\limits^{\bullet}(t)\le c\hskip 3pt\mathfrak{p}(t),\quad c\in\mathbb{R}^+, \forall t.
\end{equation}
where

\begin{equation}\label{average-square-perturbation-rate}
  \mathop{\mathfrak{p}}\limits^{\bullet}(t)=\frac{1}{2\hskip 2pt vol(B_t)}\int_{B_t}\left(\frac{\delta \xi^2}{\delta t}\right)\eta
 =\frac{1}{vol(B_t)}\int_{B_t}\left(\frac{\delta \xi}{\delta t}.\xi\right)\hskip 3pt \eta.
\end{equation}
Here $\xi$ represents the general solution of the linearized
equation $E_k[s]$ of $E_k$ at the solution $s$. Let us denote by
$c_0$ the infimum of the positive constants $c$ such that inequality
{\em(\ref{stability-inequality})} is satisfied. Then we call
$\tau_0=1/c_0$ the {\em characteristic stability time} of the
solution $V$. If $\tau_0=\infty$ means that $V$ is
unstable.\footnote{$\tau_0$ has just the physical dimension of a
time.}

Furthermore, condition {\em(\ref{stability-inequality})} is
satisfied if the operator $\frac{\delta}{\delta t}$ is self-adjoint
on the set of solutions of the linearized equation
$E_k[s]\subset J^n_n(E[s])$, where $E[s]\equiv s^*vTW$.
\end{theorem}

\begin{theorem}{\em(The extended crystal structure of the $n$-d'Alembert equation, and stability).}\label{main-1}{\em\cite{PRA10}}
{\em 1)} For the $n$-d'Alembert equation one has the following properties.

{\em(i)} The $n$-d'Alembert equation is an extended crystal PDE for
any $n\ge 2$. If $M$ is $p$-connected, $p\in\{0,1,\cdots,n-1\}$, it
becomes an extended $0$-crystal iff $\Omega_{n-1}=0$. In particular
for $n=2$ it becomes a $0$-crystal.

{\em(ii)} The $n$-d'Alembert equation is functionally stable.

{\em(iii)} Smooth regular solutions of the $n$-d'Alembert equation,
present, in general, unstabilities at finite times. However, the
$n$-d'Alembert equation can be stabilized and its stable extended
crystal PDE is its $\infty$-prolongation $((d'A)_n)_{+\infty}$.
There all smooth regular solutions are functionally stable, i.e.,
they do not present finite times unstabilities.

{\em 2)} In the case $n=2$, with $M$ non-simply connected, $(d'A)$ remains an
extended crystal PDE, but not more an extended $0$-crystal PDE. For
example if $M$ is a bidimensional torus $T^2$. This is a connected,
orientable, non-simply connected, surface. Then, $
\Omega_1^{(d'A)}\cong{\mathbb Z}_2\oplus{\mathbb Z}_2$ (For a proof
see \cite{PRA10}.) So, the d'Alembert equation on the torus is not an
extended $0$-crystal PDE, and neither an a $0$-crystal PDE. The
crystal group of such an equation is $G(2)=\mathbb{Z}\rtimes
D_4=p4m$. Its crystal dimension is $2$.

In the case $n=2$, with $M=\mathbb{R}^2$, we can built solutions
with the methods of characteristics, that are average unstable.

{\em 3)} Let us consider the $3$-d'Alembert equation on the non-simply
connected, orientable, $3$-dimensional manifold $M={\mathbb R}P^3$. In
this case one has $\Omega_2^{(d'A)}\cong{\mathbb Z}_2\oplus{\mathbb
Z}_2$. Thus this is another example where one has $(d'A)_3$ that is
an extended crystal PDE, but it cannot be an extended $0$-crystal
PDE and neither a $0$-crystal PDE. Thus this equation has the same
crystal group and crystal dimension of equation considered in above
example.
\end{theorem}

\begin{proof}
Even if these results are proved in \cite{PRA10}, let us resume their proofs here, in order to better understand the following ones.

1) (i) The $n$-d'Alembert equation $(d'A)_n\subset J{\it
D}^n(E)$ is a $n$-order PDE, formally integrable, and
completely integrable, on the trivial vector fiber bundle $\pi:E\equiv M\times\mathbb{R}\to
M$.\footnote{$(d'A)_n$ considered in this theorem is a submanifold
of  $J{\it D}^n(E)$, hence it coincides with $Z_n\cap
C_n$.} (See \cite{PRA-RAS-1}). This means that we can locally
reproduce all the results obtained for the $n$-D'Alembert equation
on ${\mathbb R}^n$. (See Refs.\cite{PRA1, PRA-RAS-1}). For any point
$q\in(d'A)_n$, passes a local solution. Furthermore, the set of
local solutions of the $n$-d'Alembert equation on $n$-dimensional
manifolds contains the set of the local functions that can be
represented in the form as
$f(x^1,\cdots,x^n)=f_1(x^2,\dots,x^n)\cdots f_n(x^1,\dots,x^{n-1})$.
 This follows directly from previous
considerations and results contained in refs.\cite{PRA1, PRA-RAS-0,
PRA-RAS-1}. Now, the set ${\frak S}ol_{loc}(d'A)_n$, $n\ge 2$, of
all local solutions of the equation: $
 {{\partial ^n\log f}\over{\partial x_n\cdots\partial
x_1}}=0$, considered on a $n$-dimensional manifold $M$, is larger
than the set of all local functions $f$ that can be represented in
the form
$f(x^1,\cdots,x^n)=f_1(x^2,\dots,x^n)\cdots,f_n(x^1,\dots,x^{n-1})$.
(See \cite{PRA-RAS-0, PRA-RAS-1}.) In the following we shall
consider the $n$-d'Alembert equation given as a submanifold
$(d'A)_n$ of the jet space $J^n_n(E)$ by means of the embedding
$(d'A)_n\hookrightarrow J{\it
D}^n(E)\hookrightarrow J^n_n(E)$. The characterization of global
solutions of $(d'A)_n$ is made by means of its integral bordism
groups. One has: $ \Omega_{p}^{(d'A)_n}\cong\Omega_{p}((d'A)_n)$,
for $ p\in\{0,\cdots,n-1\}$. This follows from the fact that the
$n$-d'Alembert equation is formally integrable, and completely
integrable. (See \cite{PRA-RAS-1}.) We get:
$\Omega_p^{(d'A)_n}\cong\underline{\Omega}_p(M)\cong\bigoplus_{r,s,r+s=p}H_r(M;{\mathbb
Z}_2)\otimes_{{\mathbb Z}_2}\Omega_s$, $p\in\{0,\cdots,n-1\}$. In
the particular case that $\dim M=2$ and $p$-connected,
$p\in\{0,1\}$, then the integral bordism group $\Omega_1^{(d'A)}=0$.
Thus $(d'A)$ is an extended $0$-crystal PDE. Furthermore, one can
also prove that for such a case there are not obstructions coming
from the integral characteristic numbers. In fact, all the
conservation laws on closed $1$-dimensional smooth integral
manifolds are zero \cite{PRA1}. Then one has $cry(d'A)=0$,  for
$p$-connected $M$, $p\in\{0,1\}$. Thus in this case $(d'A)$ becomes
a $0$-crystal.

(ii) The $n$-d'Alembert equation is functionally stable since it is formally integrable and completely integrable.
(See Theorem \ref{main}.)

(iii) The functional unstabilities come from the fact that the
symbol of the $n$-d'Alembert equation is not zero. In fact one
has
\begin{equation}\label{symbol-dimension}
    \dim(g_n)_q=\frac{(2n-1)!}{n!(n-1)!}-1, \quad \forall q\in(d'A)_n.
\end{equation}

Furthermore, in the $\infty$-prolongation
$((d'A)_n)_{+\infty}\subset J^\infty_n(W)$, we get all the smooth
solutions of $(d'A)_n$, and there, since the corresponding symbol is
zero, $((g_n))_{+\infty}=0$, do not exist admissible singular
(non-weak) perturbations. Thus, $((d'A)_n)_{+\infty}$ is necessarily
the stable extended crystal of $(d'A)_n$. Therefore, $(d'A)_n$ is a
stabilizable PDE.

2) We have proved in Refs.\cite{PRA-RAS-0, PRA-RAS-1} that $(d'A)$
admits the following characteristics strips:
\begin{equation}\label{characteristic-strips}
\left\{
\begin{array}{ll}
\zeta_1&\equiv u[\partial y+u_y\partial u+ u_{yx}\partial
u_x+u_{yy}\partial u_y]+u_{xx}u_y\partial u_{xx}+ u_{yy}u_x\partial
u_{yx}\\
  \zeta_2&\equiv u[\partial x+u_x\partial
u+ u_{yx}\partial u_y+u_{xx}\partial u_x]+u_{xx}u_y\partial u_{xy}+
u_{yy}u_x\partial u_{yy}.\\
\end{array}
\right.
\end{equation}
These generate characteristic $1$-dimensional distributions
respectively in the following sub-equations ${}_i(d'A)\subset(d'A)$,
$i=1,2$:
\begin{equation}\label{sub-equations}
{}_1(d'A):\left\{\begin{array}{l} u_{xx}=0\\
uu_{xy}-u_xu_y=0\\
\end{array}\right\};\qquad {}_2(d'A):\left\{\begin{array}{l}
u_{yy}=0\\
u u_{xy}-u_xu_y=0\\
\end{array}\right\}.
 \end{equation}

For such equations the above mentioned $1$-dimensional distributions
are respectively characteristic distribution. Therefore, for such
equations we can built characteristic solutions that are, of course,
also solutions of $(d'A)$. For example, we have proved
in \cite{PRA-RAS-1} that the solution generated by $\zeta_1$ is
given by the following formula:
\begin{equation}\label{solution}
    u(x,y)=(\frac{\beta}{2}y^2+\alpha y+1)h(x),
\end{equation}
where $\alpha,\beta\in\mathbb{R}$ and $h(x)$ is an arbitrary
function on one real variable. Let us, now, investigate, if such a
solution is average stable. The parametric equations for the
characteristic flow on such a solution, say $V\subset(d'A)$, is
given by the following differential system:
\begin{equation}\label{characteristic-system}
\left\{
\begin{array}{l}
  \dot x=0 \\
  \dot y=u\\
  \dot u=uu_y.\\
\end{array}
\right.
\end{equation}
The general solution of the linearized equation $(d'A)[V]\subset J{\it D}^2(E[s])$ can be obtained from the general symmetry vector field for $(d'A)$, given in \cite{PRA-RAS-1}. Then we get
\begin{equation}\label{general-perturbation}
\left\{
\begin{array}{l}
  \xi =[s(y)+r(x)]u\partial u\\
 u(x,y)=(\frac{\beta}{2}y^2+\alpha y+1)h(x),\\
\end{array}
\right.
\end{equation}
where $s$ and $r$ are arbitrary functions. Let us denote by $\xi(x,y)$ the component of the vertical vector field $\xi$. Then one explicitly has $\xi(x,y)=[s(y)+r(x)](\frac{\beta}{2}y^2+\alpha y+1)h(x)$. For the arbitrariness of the functions $r$, $s$ and $h$, one can see that $\xi(x,y)$ can ave singular points. So the solution (\ref{solution}) is not stable in $(d'A)$. Furthermore, it is not asymptotic stable, since $\lim_{y\to\infty}\xi(x,y)=\infty$. In order to investigate whether it is average stable, let us consider the differential operator $\frac{\delta\xi}{\delta t}$
on $(d'A)[V]$. One has $\frac{\delta\xi}{\delta t}=(\partial t.\xi)+(\partial x.\xi)\dot x+(\partial
y.\xi)\dot y=(\partial y.\xi)u(x,y)$. For its adjoint, one has
$\frac{\delta^*\phi}{\delta t}=-(\partial y.(u(x,y)\phi))=-(\partial y.\phi)u(x,y)-(\partial y.u(x,y))\phi$. Thus, the operator $\frac{\delta\xi}{\delta t}$ is not self-adjoint on the solution in (\ref{solution}), hence, such a solution is not average stable.

3) This follows directly from previous points.
\end{proof}

\begin{theorem}[Stability in exotic $8$-d'Alembert PDE's over
$\mathbb{R}^8$]\label{stability-in-exotic-8-d-alembert-pdes}
Let us consider $(d'A)_8$ over $\mathbb{R}^8$.
The integral singular bordism group $\Omega_{7,s}^{(d'A)_8}$ of the $8$-d'Alembert PDE over $\mathbb{R}^8$ is $\Omega_{7,s}^{(d'A)_8}=\mathbb{Z}_2$. If we consider admissible Cauchy manifolds $N\subset(d'A)_8$, identified with $7$-dimensional homotopy spheres, {\em(homotopy equivalence full admissibility hypothesis)}, then one has $\Omega_{7,s}^{(d'A)_8}\cong \Omega_{7}^{(d'A)_8}=0$, hence $(d'A)_8$ becomes an extended $0$-crystal PDE, but also a $0$-crystal PDE.
The bordism classes in $\Omega_{7}^{(d'A)_8}$ are identified by Cauchy manifolds represented by diffeomorphic homotopy spheres.
In particular, in the homotopy equivalence full admissibility hypothesis, starting from an admissible Cauchy manifold $N_0\subset(d'A)_8$, identified with $S^7$, one can arrive with a singular solution, to any other admissible Cauchy manifold $N_1\subset(d'A)_8$. Such a solution is unstable. Moreover, there exists a smooth solution $V$, such that $V=N_0\sqcup N_1$, iff $N_0\cong N_1$. Such a solution can be stabilized.
\end{theorem}

\begin{proof}
In fact, $\Omega_{7,s}^{(d'A)_8}\cong\bigoplus_{0\le r,s\le 7}\Omega_r\otimes_{\mathbb{Z}_2}H_s(M;\mathbb{Z}_2)$. Taking into account that for $M\cong\mathbb{R}^8$ one has $H_r(M;\mathbb{Z}_2)=0$ for $0<r\le 7$, and $H_0(M;\mathbb{Z}_2)=\mathbb{Z}_2$, and that $\Omega_7\cong\mathbb{Z}_2$, we get $\Omega_{7,s}^{(d'A)_8}=\mathbb{Z}_2$. If we consider admissible Cauchy $7$-dimensional homotopy spheres only, we have that they have necessarily all integral characteristic numbers, i.e., the evaluations on such manifolds of all the conservation laws, give same numbers. (For the proof one can copy the similar proof given in \cite{PRA17, AG-PRA1} for the Ricci flow equation.) Therefore they belong to the same singular integral bordism class, i.e.  $\Omega_{7,s}^{(d'A)_8}=0$.
Since, one has the short exact sequence (\ref{short-exact-sequence-seven-d-alembert-pde-bordism-smooth-solutions}),

\begin{equation}\label{short-exact-sequence-seven-d-alembert-pde-bordism-smooth-solutions}
 \xymatrix{0\ar[r]&K_{7,s}^{(d'A)_8}\ar[r]&\Omega_{7}^{(d'A)_8}\ar[r]&\Omega_{7,s}^{(d'A)_8}\ar[r]&0}.
\end{equation}

we get that under the {\em homotopy equivalence full admissibility hypothesis}, one has $\Omega_{7}^{(d'A)_8}\cong K_{7,s}^{(d'A)_8}$. Let us emphasize that even if the number of differentiable structures on $7$-dimensional spheres is $28$, smooth Cauchy-manifolds-exotic-7-spheres cannot be contained into $((d'A)_8)_{+\infty}$, since they are singular integral manifolds. So smooth Cauchy manifolds contained into $((d'A)_8)_{+\infty}$ can be identified with $S^7$ only.
Furthermore, taking into account that smooth solutions, bording smooth Cauchy manifolds, necessitate to identify diffeomorphisms between the corresponding sectional submanifolds, it follows that must be $\Omega_{7}^{(d'A)_8}=0$ too. Therefore, we get also $cry((d'A)_8)=0$.

For the previous arguments it is important to state that the space of conservation laws is not zero.

\begin{lemma}
The space of conservation laws of $(d'A)_n$ is not zero.
\end{lemma}

\begin{proof}
In fact, conservation laws of $(d'A)_n$ are $(n-1)$-differential forms $\omega=\omega_i\: dx^1\wedge\cdots\wedge\: \widehat{dx^i}\wedge\cdots\wedge dx^n$ on $((d'A)_n)_{+\infty}$, such that for any smooth integral $n$-manifold $V\subset (d'A)_n$, solution of $(d'A)_n$, one has $d\omega|_V=0$. Then, one can take the $(n-1)$-differential forms $\omega$, given in (\ref{conservation-laws-in-n-d-alembert-pde}).
\begin{equation}\label{conservation-laws-in-n-d-alembert-pde}
   \left\{
\begin{array}{ll}
  \omega&=\omega_i\: dx^1\wedge\cdots\wedge\: \widehat{dx^i}\wedge\cdots\wedge dx^n\\
  \omega_i&=\omega(x^1,\cdots,\widehat{x^i},\cdots,x^n,I_{\alpha\: i})_{\alpha\not=i,\: |\alpha|\ge 0}\\
  I_{\alpha\: i}&\equiv(\partial x_\alpha.(\partial x_1\cdots\widehat{\partial x_i}\cdots\partial x_n.\log f))\\
\end{array}
   \right.
\end{equation}

where $f:M\to\mathbb{R}$ is a smooth function on $M$ and $\omega_i$ are arbitrary smooth functions of their arguments. The ''widehat'' over the symbols means absence of the underlying symbols. In fact, one has the following.
\begin{equation}\label{proof-conservation-laws-in-n-d-alembert-pde}
   \left\{
\begin{array}{ll}
d\omega&=\sum_{1\le i\le n}\sum_{p\not=i}(\partial x_p.\omega_i)\: dx^p\wedge dx^1\wedge\cdots\wedge \widehat{dx^i}\cdots\wedge dx^n=0\\
&+\sum_{1\le i\le n}\sum_{\alpha\not=i;\: |\alpha|\ge 0}
(\partial I_{i\:\alpha}.\omega_i)(\partial x_i.I_{i\:\alpha})\: dx^1\wedge\cdots\wedge dx^n\\
&=\sum_{\alpha\not=i;\: |\alpha|\ge 0}[\sum_{1\le i\le n}(\partial I_{i\:\alpha}.\omega_i)](\partial x_\alpha.(\partial x_1\cdots\partial x_n.\log f))dx^1\wedge\cdots\wedge dx^n\\
 \end{array}
   \right.
\end{equation}

Now $d\omega|_V=0$ if $V\subset(d'A)_n$ is a smooth $(n+1)$-dimensional integral manifold, (singular) solution of $(d'A)_n$. In fact, if $f$ satisfies the equation $(\partial x_1\cdots\partial x_n.\log f)=0$, then $(\partial x_\alpha.(\partial x_1\cdots\partial x_n.\log f))=0$, for $|\alpha|\ge 0$. This directly follows from the prolongations of the formally integrable and completely integrable $n$-d'Alembert equation. For example for $n=2$ we get for $(d'A)_2$ and its first prolongation $((d'A)_2)_{+1}$ the equations given in (\ref{n-2-case-d-alembert-pde-and-first-prolongation}).
\begin{equation}\label{n-2-case-d-alembert-pde-and-first-prolongation}
 (d'A)_2:\hskip 3pt  \left\{f_{xy}f-f_xf_y=0\right\}; \hskip 3pt((d'A)_2)_{+1}:\hskip 3pt  \left\{
 \begin{array}{l}
   f_{xy}f-f_xf_y=0\\
   f_{xxy}f-f_{xx}f_y=0\\
   f_{xyy}f-f_{yy}f_x=0\\
 \end{array}\right\}.
\end{equation}
One other hand we have:
\begin{equation}\label{n-2-case-d-alembert-pde-and-first-prolongation-a}
\left\{
 \begin{array}{ll}
   (\partial x(\partial x\partial y.\log f))&=(f_{xxy}f-f_{xx}f_y)f-2(f_{xy}f-f_xf_y)f_x\\
   (\partial y(\partial x\partial y.\log f))&=(f_{xyy}f-f_{yy}f_x)f-2(f_{xy}f-f_xf_y)f_y\\
   \end{array}\right.
\end{equation}
Therefore, on the $2$-d'Alembert equation one has
\begin{equation}\label{n-2-case-d-alembert-pde-and-first-prolongation-b}
\left\{
 \begin{array}{ll}
  (\partial x(\partial x\partial y.\log f))|_V=0 \\
   (\partial y(\partial x\partial y.\log f))|_V=0.\\
   \end{array}\right.
\end{equation}

This process can be iterated on all the prolongation orders.
\end{proof}

To conclude the proof of Theorem \ref{stability-in-exotic-8-d-alembert-pdes} it is enough to consider that at finite order, where live singular solutions, the symbol of the $8$-d'Alembert equation is not zero. Thus these solution are unstable. Instead, smooth solutions can be stabilized, since these can be identified with smooth integral manifolds of the infinity prolongation $((d'A)_8)_{+\infty}$, where the symbol is zero.
(See Theorem \ref{criteria-fun-stab}.)
\end{proof}

\begin{cor}
In the homotopy equivalence full admissibility hypothesis, $(d'A)_8$ admits a {\em singular global attractor}, in the sense introduced in \cite{PRA16}, i.e., all the admissible Cauchy manifolds belong to the same integral singular bordism class of $(d'A)_8$. Furthermore, in the {\em sphere full admissibility hypothesis}, i.e. when we consider admissible all the smooth Cauchy manifolds identifiable via diffeomorphisms with $S^7$, $(d'A)_8$ admits a {\em smooth global attractor}, in the sense that all the smooth admissible Cauchy manifolds belong to the same integral smooth bordism class of $(d'A)_8$.
\end{cor}

\end{document}